\documentclass[letterpaper, 10 pt, conference]{ieeeconf}  
 \pdfoutput=1
\usepackage{cite}
\usepackage{amsmath,amssymb,amsfonts}
\usepackage{algorithm}
\usepackage{algorithmic}
\usepackage{graphicx}        
\usepackage{subfigure}
\usepackage{enumerate}  
\usepackage{bm}
\newtheorem{assumption}{Assumption}
\newtheorem{theorem}{Theorem}
\newtheorem{lemma}{Lemma}

\newtheorem{remark}{Remark} 

\IEEEoverridecommandlockouts                              
\title{\LARGE \bf
	A Distributed Buffering Drift-Plus-Penalty Algorithm for Coupling Constrained Optimization
}

\author{
	Dandan Wang, Daokuan Zhu, Zichong Ou, and Jie Lu,~\IEEEmembership{Member,~IEEE}
	\thanks{D. Wang is with the School of Information Science and Technology, ShanghaiTech University, Shanghai 201210, China, with the University of Chinese Academy of Sciences, Beijing 100049, China, and with Shanghai Institute of Microsystem and Information Technology, Chinese Academy of Sciences, Shanghai 200050, China. Email:
		{\tt\small wangdd2@shanghaitech.edu.cn}.}
	\thanks{D. Zhu, Z. Ou, and J. Lu are with the School of Information Science and
		Technology, ShanghaiTech University, Shanghai 201210, China.
		Email:{\tt\small\{zhudk, ouzch, lujie\}@shanghaitech.edu.cn}.}}

\begin{document}
	\maketitle
	\thispagestyle{empty}
	\pagestyle{empty}





\begin{abstract}
	This paper focuses on distributed constrained optimization over time-varying directed networks, where all agents cooperate to optimize the sum of their locally accessible objective functions subject to a coupled inequality constraint consisting of all their local constraint functions. To address this problem, we develop a buffering drift-plus-penalty algorithm, referred to as B-DPP. The proposed B-DPP algorithm utilizes the idea of drift-plus-penalty minimization in centralized optimization to control constraint violation and objective error, and adapts it to the distributed setting. It also innovatively incorporates a buffer variable into local virtual queue updates to acquire flexible and desirable tracking of constraint violation. We show that B-DPP achieves $O(1/\sqrt{t})$ rates of convergence to both optimality and feasibility, which outperform the alternative methods in the literature. Moreover, with a proper buffer parameter, B-DPP is capable of reaching feasibility within a finite number of iterations, which is a pioneering result in the area. Simulations on a resource allocation problem over 5G virtualized networks demonstrate the competitive convergence performance and efficiency of B-DPP. 
	


\end{abstract}
%

\section{Introduction}
\label{sec:introduction}


Distributed constrained optimization has gained substantial interest recently 
owing to its widespread applications in various networked systems, such as task offloading in wireless network\cite{xiao2018distributed}, resource allocation in power system\cite{vujanic2016decomposition}, and target tracking in multi-robot system\cite{huang2023distributed}. In these scenarios, individual agents often tangle with others via coupled constraints to satisfy shared resource allocation or global requirements. All agents cooperate only with neighbors to minimize the sum of local objective functions subject to both coupled constraints and local set constraints. 



So far, a variety
of distributed optimization algorithms have been developed to handle coupled constraints, including the  integrated primal-dual
proximal algorithm\cite{wu2022distributed}, extensions of the dual subgradient method with double averaging\cite{liu2020unitary,liang2019Cyber,bose2021distributed}, consensus-based dual subgradient methods \cite{simonetto2016primal, liang2019distributed,falsone2017dual}, primal-dual algorithms based on the saddle-point method \cite{chang2014distributed,mateos2016distributed,lee2017sublinear,li2020distributed}, and a distributed algorithm based on a relaxation of the primal problem \cite{notarnicola2019constraint} and its extension to time-varying graphs\cite{camisa2021distributed}. Nevertheless, the works \cite{liu2020unitary,liang2019Cyber,bose2021distributed,liang2019distributed, simonetto2016primal, wu2022distributed, notarnicola2019constraint} cannot be applied to time-varying directed networks, and  \cite{chang2014distributed,falsone2017dual,notarnicola2019constraint,liang2019distributed} only achieve asymptotic convergence without guaranteed convergence rates. Although \cite{mateos2016distributed, camisa2021distributed,li2020distributed,lee2017sublinear} provide explicit convergence rates for distributed constraint-coupled optimization over time-varying networks, \cite{camisa2021distributed,li2020distributed} suffer from convergence rates slower than $O(1/\sqrt{t})$, \cite{mateos2016distributed, camisa2021distributed} require all agents know a Slater point, which is unavailable in most practical systems, and \cite{lee2017sublinear} needs a restrictive assumption on the boundedness of dual variables.

It is also worth noting that the existing works\cite{liu2020unitary,liang2019Cyber,bose2021distributed,simonetto2016primal,notarnicola2019constraint,liang2019distributed,chang2014distributed,falsone2017dual,mateos2016distributed,camisa2021distributed,wu2022distributed,lee2017sublinear,li2020distributed,yu2015convergence} only have asymptotic feasibility guarantees when addressing coupled nonlinear constraints. However, in practice, one can only run the algorithms for a finite number of iterations, and thus the algorithms possibly output infeasibile solutions. This may cause system breakdown, especially for resource allocation problems in real-world systems whose constraints often have physical meanings\cite{wu2023distributed}.


Motivated by the above challenges, this paper proposes a buffering drift-plus-penalty (B-DPP) algorithm  to solve distributed coupling constrained optimization with additive global objective and inequality constraint functions over time-varying directed networks. Our design of B-DPP is initially inspired from the centralized drift-plus-penalty (DPP) algorithm, which characterizes optimality and feasibility via analyzing a drift-plus-penalty term and tracks infeasibility via a global virtual queue update. Instead, our proposed B-DPP introduces a \textit{local} virtual queue variable for each agent to track cumulative constraint violations, and makes local decisions to minimize an upper bound on a drift-plus-penalty term, so that the constraint violation and the global objective error can be controlled in a fully distributed fashion. Furthermore, B-DPP integrates a buffer parameter with the local virtual queue updates to enhance flexibility and achieve desirable effect of  constraint violation tracking, despite the lack of centralized coordination required by DPP. Our contributions are highlighted as follows:
\begin{enumerate}
	\item B-DPP is a novel distributed algorithm for solving \textit{coupling constrained} optimization over \textit{time-varying directed} networks, while the existing methods \cite{liu2020unitary,liang2019Cyber,bose2021distributed,liang2019distributed, simonetto2016primal, wu2022distributed, notarnicola2019constraint} are limited to fixed underlying networks.
	\item B-DPP is memory-wise and computationally efficient, in which each agent only maintains a local decision variable, a virtual queue variable, and a running average of its local decision variables. Thus, B-DPP requires less memory to store the variables and enjoy lower computation cost compared with the related methods in \cite{liu2020unitary,liang2019Cyber,bose2021distributed,simonetto2016primal,liang2019distributed,chang2014distributed,mateos2016distributed,wu2022distributed,lee2017sublinear,li2020distributed}.
	\item B-DPP is shown to achieve $O(1/\sqrt{t})$ ergodic convergence rates in terms of both objective error and constraint violation under mild assumptions, which outperform those of the alternative methods \cite{liu2020unitary,liang2019Cyber,bose2021distributed,simonetto2016primal,notarnicola2019constraint,liang2019distributed,chang2014distributed,falsone2017dual,mateos2016distributed,camisa2021distributed, lee2017sublinear,li2020distributed}.
	\item 
	With a proper buffer parameter, B-DPP is shown to reach \textit{feasibility} within a finite number of iterations. To the best of our knowledge, there have been no existing distributed algorithms that ensure such a result.
\end{enumerate}
The paper unfolds as follows.
Section~\ref{sec:problemformulation} formulates a distributed constraint-coupled problem over time-varying directed networks. Section~\ref{sec:B-DPPalgorithm} develops the proposed B-DPP
algorithm, and Section~\ref{sec:MainResults} provides convergence analysis of B-DPP. Section~\ref{sec: numericalexperiments} presents the numerical
experiments, and Section~\ref{sec:conclusion} concludes the paper.

\addtolength{\topmargin}{0.045in}
\section{Problem Formulation}
\label{sec:problemformulation}
Consider a time-varying directed network with a set $\mathcal{V}=\{1, 2, \ldots, N\}$ of agents. Each agent $i \in \mathcal{V}$ privately possesses a local objective function $f_{i}:\mathbb{R}^{d_i}\rightarrow \mathbb{R}$, a constraint function $g_{i}:\mathbb{R}^{d_i}\rightarrow \mathbb{R}^p$, and a local constraint set $X_i\subseteq\mathbb{R}^{d_i}$.  All the agents cooperatively solve the following constraint-coupled problem under Assumption~\ref{asm:probassumption}:
\begin{equation} \label{eq:primalprob}
	\begin{array}{cl}
		\underset{\substack{x_i, \forall i \in \mathcal{V} }}{\operatorname{minimize}} ~\!\!\!& f(\mathbf{x}):=\sum_{i=1}^{N} f_{i}(x_{i})\\  
		{\operatorname{ subject~to }}\!\!\! & g (\mathbf{x}):=\sum_{i=1}^{N}g_{i}(x_{i}) \leq \mathbf{0}_p,\\\vspace{2mm} 
		{} & {x_i \in X_i,\;\forall i \in \mathcal{V}},
	\end{array}
\end{equation} 
where $\mathbf{x}=[(x_1)^T, \ldots, (x_N)^T]^T\in \mathbb{R}^{\sum_{i=1}^{N}d_i}$ and $X=X_1\times \cdots \times X_N$ is the Cartesian product of all the $X_i$'s.
\begin{assumption}\label{asm:probassumption}
	Problem~\eqref{eq:primalprob} satisfies the following:
	\begin{enumerate}[(1)]
		\item For each $i\in\mathcal{V}$, $X_i$ is a compact convex set with diameter $R:= \sup_{x_i,\tilde{x}_i \in X_i} \|x_i-\tilde{x}_i\| <\infty$. \label{compactset}
		\item For each $i\in\mathcal{V}$, $ f_{i}$ and $g_i$ are convex on $X_{i}$. \label{convexassumption} 
		\item There exists at least one optimum $\mathbf{x}^*=[(x_1^*)^T, \ldots, (x_N^*)^T]^T$ of problem~\eqref{eq:primalprob}.\label{least1soluprimal}
		\item There exists a constant $\epsilon >0$ and a point $\hat{x}_i \in\operatorname{relint}(X_i),~\forall i \in \mathcal{V}$ such that $\sum_{i=1}^{N}g_{i}(\hat{x}_{i}) \leq-\epsilon\mathbf{1}_p$. \label{slatercondition} 
	\end{enumerate}
\end{assumption}

Assumption~\ref{asm:probassumption}\eqref{compactset}--\eqref{convexassumption} indicate that there exist constants $F > 0$,  $G>0$ such that for any $x_i$, $\tilde{x}_{i} \in X_i$, $ i \in \mathcal{V}$, $\|g_i(x_i)\|\le F$, $\| f_{i}(x_{i})-f_{i}(\tilde{x}_{i})\|\le G$. Assumption~\ref{asm:probassumption}\eqref{least1soluprimal} ensures that problem~\eqref{eq:primalprob} is solvable and Assumption~\ref{asm:probassumption}\eqref{slatercondition} is the Slater's condition that will be used to reduce infeasibility.

We model the network as the time-varying directed graph $\mathcal{G}_t=(\mathcal{V}, \mathcal{E}_t)$, where $t$ denotes time instance and at each time $t \ge 0$, $\mathcal{E}_t\subseteq\{\{i,j\}:i,j\in\mathcal{V},\;i\neq j\}$ represents the set of edges, i.e., agent $i$ can receive messages from agent $j$ if and only if there exists a directed link $(j,i) \in  \mathcal{E}_t$. Let $\mathcal{N}_{i,t}^{\text{in}}=\{j | (j,i) \in  \mathcal{E}_t\}\cup \{i\}$ and $\mathcal{N}_{i,t}^{\text{out}}=\{j|(i,j) \in  \mathcal{E}_t\}\cup\{i\}$ be the sets of in-neighbors and out-neighbors of node $i$. We impose the following connectivity assumption on $\mathcal{G}_t$, which does not require the network to be strongly connected at every time instant and allows each agent to influence all the others within every time interval of length $B$. 
\begin{assumption}[\textit{B-connectivity}]\label{Bconnect}
	There exists an integer $B \ge1$ such that for any $k \ge 0$, the graph $(\mathcal{V}, \bigcup_{t=kB+1}^{(k+1)B}\mathcal{E}_t)$ is strongly connected. 
\end{assumption}

We associate $\mathcal{G}_t$ with a time-varying mixing matrix $W_t$ for information fusion, whose $(i,j)$-entry at time $t$ is given by $w_{ij,t} >0$ if $(j,i) \in  \mathcal{E}_t$ or $i=j$, and $w_{ij,t}=0$, otherwise. 
\begin{assumption}\label{ass: networkassumption}
	The mixing matrices $W_t$, $\forall t \ge 0$ satisfy the following:
	\begin{enumerate}[(1)]
		\item There exists a constant $a \in (0,1)$ such that for each $t \ge 0$, $w_{ij,t}>a$ if $w_{ij,t} >0$. \label{graphnondegeneracy}
		\item  $W_t$, $\forall t\ge 0$ is doubly stochastic, i.e., $\sum_{i=1}^{N} w_{ij,t}=\sum_{j=1}^{N} w_{ij,t}=1$, $\forall i,j \in \mathcal{V}$.\label{doublystochastic}
	\end{enumerate}
\end{assumption}

Assumption~\ref{ass: networkassumption} is a standard assumption that is also adopted in \cite{chang2014distributed,falsone2017dual,mateos2016distributed,camisa2021distributed,lee2017sublinear,li2020distributed}. Assumption~\ref{ass: networkassumption} \eqref{doublystochastic} requires balanced communications between agents, which can be simply satisfied if we allow  bidirectional communication between agents and enforce symmetry on the mixing matrix.
\section{Algorithm Development}\label{sec:B-DPPalgorithm}

In this section, we develop a novel distributed algorithm to solve problem \eqref{eq:primalprob} over time-varying directed networks.


We first consider applying the drift-plus-penalty (DPP) algorithm \cite{yu2015convergence} to solve problem \eqref{eq:primalprob}. It introduces virtual queue variables for inequality constraints to record cumulative constraint violations and derives upper bounds on objective error and constraint violation based on the analysis of a drift-plus-penalty term, where the drift involves the change of virtual queues and the penalty term relies on the objective function.
Let $\mu_t \in \mathbb{R}^p$ be the virtual queue variable for the coupled inequality constraint $g(\mathbf{x})=\sum_{i=1}^{N}g_{i}(x_{i})\le \mathbf{0}_p$ and $\mathbf{x}_t=[(x_{1,t})^T, \ldots, (x_{N,t})^T]^T \in \mathbb{R}^{\sum_{i=1}^{N}d_i}$ be the estimate of the optimal solution $\mathbf{x}^*$ at time $t$. DPP is described as follows:  For arbitrarily given $\mu_0 \ge \mathbf{0}_p $ and each $ t \ge 0$,
\begin{align}
	\mathbf{x}_{t+1}&=\operatorname{\arg min}_{\mathbf{x} \in X}\{Vf(\mathbf{x})+\langle\mu_t,g(\mathbf{x})\rangle\}, \displaybreak[0]\label{eq:xtdpp}\\
	\mu_{t+1}&=\max\{\mu_t+g(\mathbf{x}_{t+1}), \mathbf{0}_p\}, \displaybreak[0]\label{eq:mutdpp}
\end{align}
where $V$ is a positive constant parameter used to balance the objective minimization and the infeasibility reduction. 

In \eqref{eq:mutdpp}, $\mu_{t+1}$ can be viewed as the queue backlog of constraint violation at time $t+1$. Once a constraint violation occurs at time $t+1$, i.e., $g(\mathbf{x}_{t+1}) > \mathbf{0}_p$, $\mu_{t+1}$ is equal to the queue backlog at time $t$, i.e., $\mu_t$, plus $g(\mathbf{x}_{t+1})$. By adding \eqref{eq:mutdpp} from $k=0$ to $t-1$, we obtain $\mu_t \ge \sum_{k=1}^{t}g(\mathbf{x}_{k})$, which suggests that the virtual queue $\mu_t$ can be used to track the cumulative constraint violations. From \eqref{eq:xtdpp}, a small or zero queue backlog $\mu_t$ may cause the objective function to take a dominant role when updating $\mathbf{x}_{t+1}$, possibly followed by an increase in constraint violation, while a large $\mu_t$ can amplify the importance of the constraint when updating $\mathbf{x}_{t+1}$, potentially decreasing the subsequent constraint violation. Therefore, the reductions of objective value and infeasibility alternate with each other in DPP.

Although DPP is applicable to solve problem \eqref{eq:primalprob}, it cannot be executed in a distributed manner due to the existence of the coupled inequality constraint $\sum_{i=1}^{N}g_{i}(x_{i})$. Furthermore, from the perspective of convergence analysis, the upper bound on its constraint violation only depends on the value of $\mu_t$, which may be very loose when $\mu_t$ is large. To tackle these two issues, below we propose a new distributed algorithm called \textit{ buffering drift-plus-penalty (B-DPP) algorithm}. 
\addtolength{\topmargin}{0.045in}
\begin{algorithm} [t]
	\caption{ B-DPP }
	\label{alg:offlineAlgorithmgixi}
	\begin{algorithmic}[1] 
		\STATE {\textbf{Initialization}}:  Each node $i \in \mathcal{V}$ arbitrary selects $x_{i,0} \in X_i$, $\mu_{i, 0}=\mathbf{0}_p$.
		\FOR {$t=0,1,2,\ldots$} 
		\STATE	Each node $j \in \mathcal{V}$ sends its local information $\mu_{j,t}$ to every out-neighbor $i \in \mathcal{N}_{j,t}^{\text{out}}$. 
		\STATE After receiving $\mu_{j,t}$ from each in-neighbor $j \in \mathcal{N}_{i,t}^{\text{in}}$, each node $i \in \mathcal{V}$ computes $\hat{\mu}_{i, t}$ according to \eqref{eq: updateofhatmuit}.
		\STATE Each node $i \in \mathcal{V}$ updates $x_{i,t+1}$ according to \eqref{eq:offlineupdateofxt}.
		\STATE Each node $i \in \mathcal{V}$ updates $\mu_{i,t+1}$ according to \eqref{eq:offlineupdateofuk}.
		\ENDFOR
	\end{algorithmic}
\end{algorithm} 

At each time $t\ge0$, each agent $i \in \mathcal{V}$ maintains a primal variable $x_{i,t}\in \mathbb{R}^{d_i}$ as its estimate of $x_{i}^*$ and a local virtual queue variable $\mu_{i,t} \in \mathbb{R}^{p}$ associated with its local constraint function $g_i(x_i)$ at time $t$. Starting from any $x_{i,0} \in X_i$, $\mu_{i,0}= \mathbf{0}_p $, each agent $i \in \mathcal{V}$ updates at time $ t \ge 0$ according to  
\begin{align}
	\hat{\mu}_{i, t}&=\sum\limits_{j \in \mathcal{N}_{i,t}^{\text{in}}} w_{i j, t} \mu_{j, t},\displaybreak[0]\label{eq: updateofhatmuit}\\
	x_{i,t+1}&=\arg\min\limits_{x_{i} \in X_i } \left\{ V_{t+1} f_{i}(x_{i})+\langle \hat{\mu}_{i, t}, ~ g_i(x_{i})\rangle \right.\notag \displaybreak[0]\\
	&~~~~~~~~~~~~~~~~\left.+\eta_{t+1}\|x_i-x_{i,t}\|^2\right\}, \label{eq:offlineupdateofxt} \\
	\mu_{i,t+1}&=\max\{\hat{\mu}_{i, t}+g_i(x_{i,t+1}), \mathbf{0}_p\}+\gamma_{t+1}\mathbf{1}_p \label{eq:offlineupdateofuk},
\end{align}
where $V_{t+1}>0$ is a trade-off parameter to balance objective minimization and infeasibility reduction, $\eta_{t+1}>0$ is the stepsize, and $\gamma_{t+1} >0$ is called the buffer parameter whose role will be explained shortly. 

Different from DPP in \eqref{eq:xtdpp}--\eqref{eq:mutdpp}, our proposed B-DPP algorithm associates each agent $i$'s local virtual queue $\mu_{i,t}$ with its local constraint function $g_i(x_i)$ instead of the global constraint $g(\mathbf{x})=\sum_{i=1}^{N}g_{i}(x_{i})$. In addition, we introduce an auxiliary variable $\hat{\mu}_{i,t}$ that fuses the virtual queues of agent $i$ and its in-neighbors to update $x_{i,t+1}$ and $\mu_{i,t+1}$, enabling distributed computation. Moreover, the proximal term $\eta_{t+1}\|x_i-x_{i,t}\|^2$ endows the minimized function in \eqref{eq:offlineupdateofxt} with strong convexity, which makes \eqref{eq:offlineupdateofxt} well-posed and contributes to convergence.

To understand the function of the buffer parameter $\gamma_{t}$, first note that \eqref{eq:offlineupdateofuk} guarantees $\mu_{i,t} \ge \gamma_{t}\mathbf{1}_p$, i.e., $\gamma_{t}$ can be viewed as a buffer for the local virtual queues. Thus, by setting $\gamma_t$ to be sufficiently away from zero at the beginning and gradually diminishing with time, we obtain extra safety in avoiding overly optimistic decisions that are caused by zero or tiny $\mu_{i,t}$ and would yield substantial constraint violation subsequently. Moreover, as is shown in the following lemma, $\gamma_{t}$ enables a tunable bound on constraint violation.

\begin{lemma}\label{lem:cumulativeconstraintvion}
	Suppose Assumptions \ref{asm:probassumption}--\ref{ass: networkassumption} hold. For any $t \ge 1$,
	\begin{align}
		\sum_{k=1}^t\sum_{i=1}^{N}g_i(x_{i,k})+N\sum_{k=1}^t\gamma_{k}\mathbf{1}_p
		\le \sum_{i=1}^{N}\mu_{i,t}. \displaybreak[0] \label{eq:cumulativeconsVio}
	\end{align}
\end{lemma}

\begin{proof}
	Based on \eqref{eq:offlineupdateofuk} and Assumption~\ref{ass: networkassumption},	$\sum_{i=1}^{N}\mu_{i,k+1}  \ge \sum_{i=1}^{N}\mu_{i,k}+\sum_{i=1}^N g_i(x_{i,k+1})+N\gamma_{k+1}\mathbf{1}_p$, $\forall k \ge 0$ holds. Adding this inequality from $k=0$ to $t-1$ yields \eqref{eq:cumulativeconsVio}.
\end{proof}

Lemma~\ref{lem:cumulativeconstraintvion} gives an upper bound on the cumulative constraint violation $\sum_{k=1}^t\sum_{i=1}^{N}g_i(x_{i,k})$, $\forall t \ge 1$, which depends on both the sum of the local virtual queues $\mu_{i,t}$, $\forall i\in\mathcal{V}$ and the history of the buffer parameter $\gamma_{t}$. Hence, unlike the centralized DPP algorithm \eqref{eq:xtdpp}--\eqref{eq:mutdpp} that only utilizes the global virtual queue to track infeasibility, our proposed algorithm \eqref{eq: updateofhatmuit}--\eqref{eq:offlineupdateofuk} allows for additional degree of freedom and, thus, has the potential for achieving desirable tracking effect in the decentralized scenario.


\addtolength{\topmargin}{0.045in}
\begin{table*}[t]
	\centering
	\footnotesize	\caption{{\upshape} Comparison with Related Works that Are Applicable to Time-Varying Directed Networks. Here, $\surd$ means the condition is required. }\label{table:comparison02} 
	\vspace*{0.1in}
	\resizebox{1\textwidth}{!}{ 
		\begin{tabular}{|p{3.5cm}|c|c|c|c|c|c|c|}
			\hline & \cite{falsone2017dual}&\cite{chang2014distributed}&\cite{mateos2016distributed}&\cite{lee2017sublinear}&\cite{li2020distributed}&\cite{camisa2021distributed}& B-DPP \\
			\hline Knowledge of a Slater point&  & $\surd$  &$\surd$& &&$\surd$&\\
			\hline Objective error  &asymptotic&asymptotic&$O(\frac{1}{\sqrt[]{t}})$ &$O(\frac{1}{\sqrt{t}})$&$O(\frac{1}{\sqrt[6]{t}})$&$O(\frac{1}{\ln t})$&$O(\frac{1}{\sqrt{t}})$ \\
			\hline Constraint violation &asymptotic&asymptotic&$O(\frac{1}{\sqrt[]{t}})$&$O(\frac{1}{\sqrt[]{t}})$&$O(\frac{1}{\sqrt[6]{t}})$&asymptotic&$O(\frac{1}{\sqrt[]{t}})$ / \textit{zero}\\
			\hline
			Numbers of variables that need to be stored locally  & $2 d_i+p$&$3 d_i\!+\!2p\!+\!1$& $2d_i+2p$&$2d_i+2p$&$2d_i+2p$&$d_i\!+\!2p\!+\!1$&$2d_i+p$\\
			\hline
	\end{tabular}}
\end{table*}


Below, we describe how B-DPP controls the constraint violation and the objective value through the update \eqref{eq:offlineupdateofxt}. Let $\bar{\mu}_t=\sum_{i=1}\mu_{i,t}$ and 
let $\triangle_t=\frac{1}{2}\|\bar{\mu}_{t+1}\|^2-\frac{1}{2}\|\bar{\mu}_t\|^2$ be a Lyapunov drift describing the evolution of the local virtual queues. A large $\triangle_t$ implies that a severe constraint violation possibly occurs at time $t+1$ according to \eqref{eq:offlineupdateofuk}. The updates of $x_{i,t+1}$, $\forall i \in \mathcal{V}$ in \eqref{eq:offlineupdateofxt} intend to minimize an upper bound on the drift-plus-penalty term $\triangle_t+V_{t+1}\sum_{i=1}^{N}f(x_{i,t+1})$. This can be seen from the following inequality:

\begin{align}
	&\underbrace{\triangle_t}_{\text{drift}}\!+\!\underbrace{V_{t+1}\sum\limits_{i=1}^{N}f_i(x_{i,t+1})}_{\text{penalty}}\le (F\!+\!2\sqrt{p}\gamma_{t+1})\!\sum_{i=1}^{N}\|\bar{\mu}_t\!-\!\hat{\mu}_{i,t}\|\displaybreak[0]\notag\\ 
	&+\!\! \underbrace{\sum_{i=1}^{N}\!V_{t+1}f_i(x_{i,t+1})\!+\!\langle\hat{\mu}_{i, t}, ~ \!g_i(x_{i,t+1})\rangle\!\!+\!\eta_{t+1}\|x_{i,t+1}\!\!-\!\!x_{i,t}\|^2}_{\text{the minimized function in \eqref{eq:offlineupdateofxt}}}\displaybreak[0]\notag\\
	&+2NF^2+4NFp\gamma_{t+1}\!+\!4Np\gamma_{t+1}^2+N\langle\bar{\mu}_t, \gamma_{t+1}\mathbf{1}_p\rangle,\displaybreak[0] \label{eq:intuitionofCDPP}
\end{align}
which is proved in Appendix~\ref{prf:boundoftriangle}. Note that on the right-hand side of \eqref{eq:intuitionofCDPP}, the values of $x_{i,t}$ and $\mu_{i,t}$ are already obtained upon completing the previous time instance.
Hence, the primal decisions $x_{i,t+1}$, $\forall i \in \mathcal{V}$ are capable of controlling the constraint violation and the global objective value through minimizing the above bound on the drift-plus-penalty term.



\begin{remark}\label{rmk:cdppcomputationcommu}
	B-DPP only requires each agent to exchange its own local virtual queue with its neighbors without leaking any information about its primal variable (i.e., its estimate on the global optimum), which secures the privacy of agents. Compared with \cite{wu2022distributed,liu2020unitary,liang2019Cyber,bose2021distributed,simonetto2016primal,liang2019distributed,chang2014distributed,mateos2016distributed,lee2017sublinear,li2020distributed} that need to maintain variables of dimension $> 2d_i+p$, B-DPP only needs each agent to store $2d_i+p$ variables, consuming less memory.
	\begin{remark}
		The time-varying parameters in B-DPP can be set as $V_t=\sqrt
		t$, $\eta_t=t$, and $\gamma_t =C/\sqrt{t}$, $\forall C>0$ according to  \eqref{eq:stepsizechosen} in next section. The parameter setting of B-DPP is straightforward and does not depend on the knowledge of a Slater point like \cite{chang2014distributed,simonetto2016primal,camisa2021distributed,mateos2016distributed}, require square summable stepsizes like \cite{falsone2017dual}, or need to know the time horizon $T$ in advance like \cite{bose2021distributed}. 
		
	\end{remark}

\end{remark}

\section{Convergence Analysis}
\label{sec:MainResults}
In this section, we show that the proposed B-DPP algorithm  achieves sublinear convergence to objective  error and is able to have \textit{zero constraint violation} within \textit{finite time} under proper parameter settings.


To present our convergence results, let $\bar{x}_{i,t}:=\frac{1}{t}\sum_{k=1}^tx_{i,k}$, $ i\in\mathcal{V}$, $t \ge 1$ be the running average of the primal variable $x_{i,k}$. The following theorem provides convergence rates with respect to the objective error and the constraint violation  at $\bar{x}_{i,t}$, $\forall i \in \mathcal{V}$.

\begin{theorem}\label{thm:convergencereuslts} 
	Suppose Assumptions \ref{asm:probassumption}--\ref{ass: networkassumption} hold. If
	\begin{equation} \label{eq:stepsizechosen}
		V_t=\sqrt{t}, \; \eta_t=t, \; \gamma_t=\frac{C}{\sqrt{t}}, \; \forall C> 0,
	\end{equation} 
	then for any $t \ge 1$, 
	\begin{align}
		&\sum_{i=1}^{N}f_i(\bar{x}_{i,t})-\sum_{i=1}^{N}f_i(x_i^*) \le \frac{C_f}{\sqrt{t}},\displaybreak[0] \label{eq:functionconvergence}\\
		&	\sum_{i=1}^{N}g_i(\bar{x}_{i,t}) \le \frac{C_g\mathbf{1}_p}{\sqrt{t}}, \displaybreak[0] \label{eq:sqrtTconstrviolation}
	\end{align}
	where $C_f$ and $C_g$ are positive constants depending on $N$, $B$, $F$, $G$, $R$, $p$, $a$, and $C$. In addition, there exist $C_0>0$ and $t_1>0$ such that for any $C \ge C_0$,
	\begin{align}
		\sum_{i=1}^{N}g_i(\bar{x}_{i,t}) \le \mathbf{0}_p, ~\forall t \ge t_1, \displaybreak[0] \label{eq:zeroconstriant}
	\end{align} 
	where $C_0$ depends on $N$, $B$, $F$, $G$, $R$, $p$, $a$ and $t_1$ also depends on these constants along with $C$. The derivations of $C_f$, $C_g$, $C_0$, and $t_1$ are given in Appendix~\ref{prf:convergenceresults}.
\end{theorem}
\begin{proof}
	See Appendix~\ref{prf:convergenceresults}.
\end{proof}

Theorem~\ref{thm:convergencereuslts} shows that B-DPP achieves $O(1/\sqrt{t})$ ergodic convergence rates in terms of optimality and feasibility under mild parameter conditions. Table~\ref{table:comparison02} compares the convergence results of B-DPP and the alternative distributed methods that are also applicable to time-varying directed networks, in which  \cite{falsone2017dual,chang2014distributed,mateos2016distributed,lee2017sublinear,li2020distributed} achieve ergodic convergence like B-DPP. Specifically,  \cite{falsone2017dual,chang2014distributed} only achieve asymptotic convergence and \cite{li2020distributed} has worse convergence rates than ours. Although \cite{mateos2016distributed, lee2017sublinear} have $O(1/\sqrt{t})$ convergence rates, \cite{mateos2016distributed} requires that all agents know a Slater point and \cite{lee2017sublinear} imposes a strict assumption on the boundedness of dual variables while B-DPP does not. In comparison with \cite{camisa2021distributed} that has non-ergodic convergence, B-DPP enjoys faster convergence to objective error and additionally provides the convergence rate with respect to the constraint violation. Besides, \cite{camisa2021distributed} also requires the knowledge of a Slater point to satisfy its convergence assumption.
As is shown Table~\ref{table:comparison02},  \cite{chang2014distributed,mateos2016distributed,lee2017sublinear,li2020distributed} need to maintain more variables than B-DPP and thus consume more memory.

Notably, \eqref{eq:zeroconstriant} implies that with a proper $\gamma_t$, B-DPP achieves zero constraint violation within a finite number of iterations. To the best of our knowledge, this is the first result of guaranteed finite-time feasibility in distributed optimization with a coupled inequality constraint. It indicates that B-DPP may quickly reach feasibility, which is important to the real-world networks when addressing problems like resource allocation.



\begin{remark}\label{rmk:Cstepsizetradoff}
	From Appendix~\ref{prf:convergenceresults},	the constant $C_0$ in Theorem~\ref{thm:convergencereuslts}  only involves the network information $N, B, a$ and the problem information $ F, G, R, p$ whose values are known. Thus, it is almost effortless to calculate $C_0$ and set $C \ge C_0$. Likewise, the time length needed to reach feasibility, i.e., $t_1$, can also be easily obtained since it only involves the above mentioned quantities and $C$.  
\end{remark}
\begin{remark}\label{rmk:Ctradeoff}
	The constant $C$ that determines the buffer parameter $\gamma_t$ leads to a trade-off between
	objective error and constraint violation. On one hand, when $C$ is sufficiently large, according to \eqref{eq:zeroconstriant} the constraint violation becomes zero after $t_1$ iterations, but the objective error bound given in \eqref{eq:functionconvergence} would be substantial since the constant $C_f$ is proportional to $C$ as is shown in Appendix~\ref{prf:convergenceresults}. On the other hand, when $C$ is smaller than the given constant $C_0$, the objective error bound in \eqref{eq:functionconvergence} can be small, but the constraint violation can only be bounded by a diminishing positive sequence and cannot be guaranteed to reach zero in finite time according to \eqref{eq:sqrtTconstrviolation}.   
\end{remark}

\addtolength{\topmargin}{0.045in}
\section{Numerical Experiment}\label{sec: numericalexperiments}
Consider a 5G virtualized network system consisting of $N$ slices, each of which is associated with $M$ virtual network functions (VNF) distributed geographically in $K$ data centers (DC). Each VNF provides wireless services in its slice to ensure wireless network access to the user attached to that slice. Each DC offers $l$ types of resources (i.e., CPU, bandwidth, and storage), with capacity limits denoted as $R_k \in \mathbb{R}^l$  for each DC $k$, $\forall k \in \{1,2,\dots, K\}$. Assume that each VNF $ m$, $\forall m \in \{1, \dots, M\}$ of the slice $i$, $\forall i \in \{1,\dots,N\}$ requires the resource  $d_i^{k,m} \in \mathbb{R}^l$ from DC $k$ to process a unit of wireless service request. Let $x_i$ be the slice thickness variable that represents the amount of wireless service for slice $i$ and thus $x_id_i^{k,m}$ is the amount of resources allocated to slice $i$\cite{halabian2019distributed}. The resource allocation problem in 5G virtualized networks
aims to determine the optimal amount of resources allocated to each slice such that the sum of cost functions of slice thicknesses is minimized while satisfying the resource constraints, which can be formulated as:
\begin{equation} \label{eq:numericalprob}
	\begin{array}{cl}
		\underset{\substack{x_i, \forall i \in \mathcal{V} }}{\operatorname{minimize}} ~\!\!\!& \sum_{i=1}^N f_i\left(x_i\right)=\frac{1}{2} \sum_{i=1}^N\left(x_i-a_i\right)^2\\  
		{\operatorname{ subject~to }}\!\!\! & \sum_{i=1}^N \sum_{m=1}^M x_i d_i^{k, m} \leq R_k, \forall k ,\\\vspace{2mm} 
		{} &  0\leq x_i \leq 2, \quad \forall i.
	\end{array}
\end{equation} 
Set $N=10$ and $l=K=M=1$. The 5G virtualized network is modeled as a time-varying directed graph with $B$-connectivity. Let $B=4$. We randomly generate $a_i \in [1,2]$, $d_i^{1,1} \in [0.5,1]$, $\forall i \in \{1, \dots, 10\}$, and $R_1 \in [5,20]$.
\begin{figure}[t] 	
	\subfigure[]
	{
		\begin{minipage}{0.466\linewidth}
			\centering 
			\includegraphics[height=0.75\linewidth,width=\linewidth]{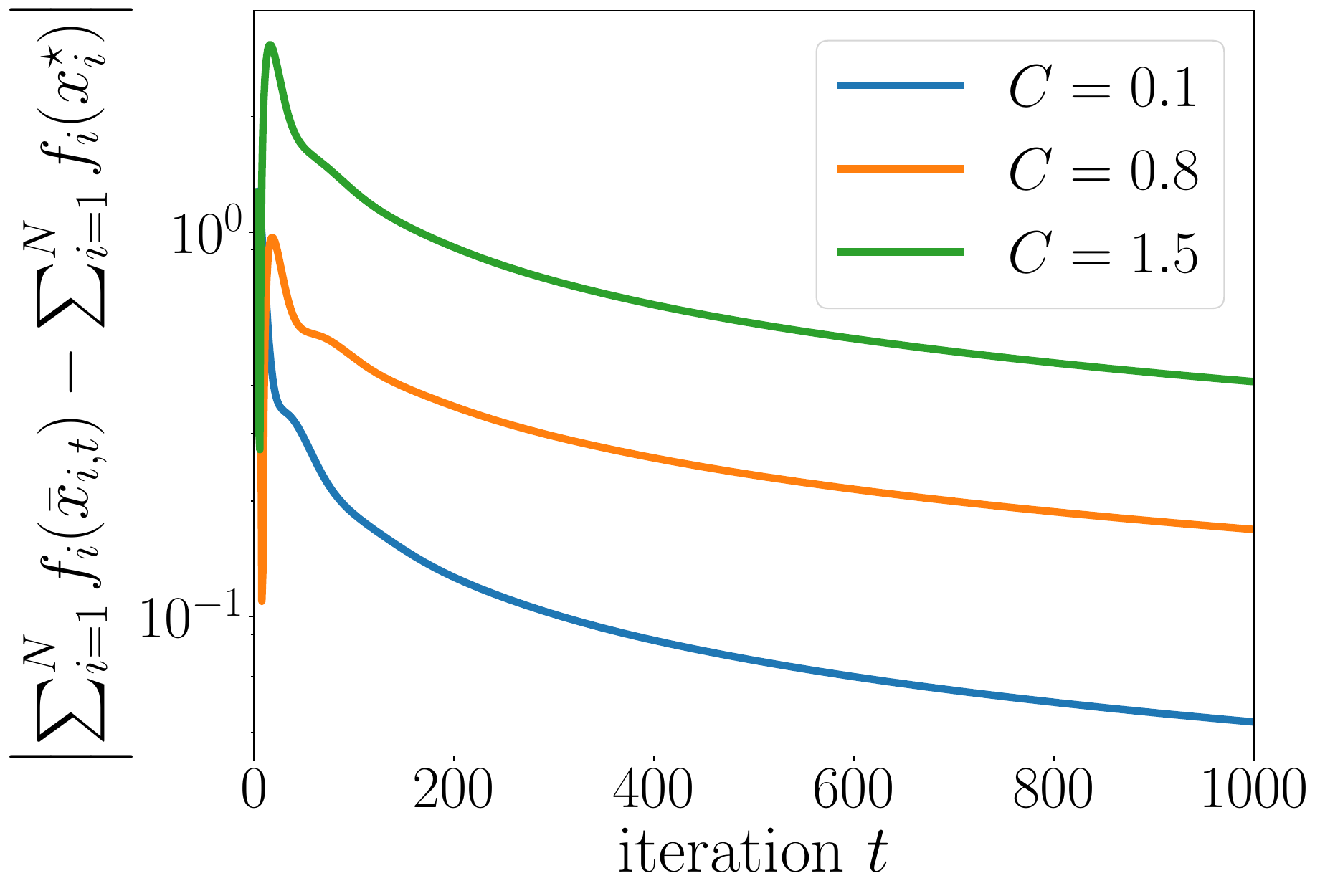}
			\label{cooperation}
		\end{minipage}
	}	
	\subfigure[]
	{
		\begin{minipage}{0.466\linewidth}
			\centering      
			\includegraphics[height=0.75\linewidth,width=\linewidth]{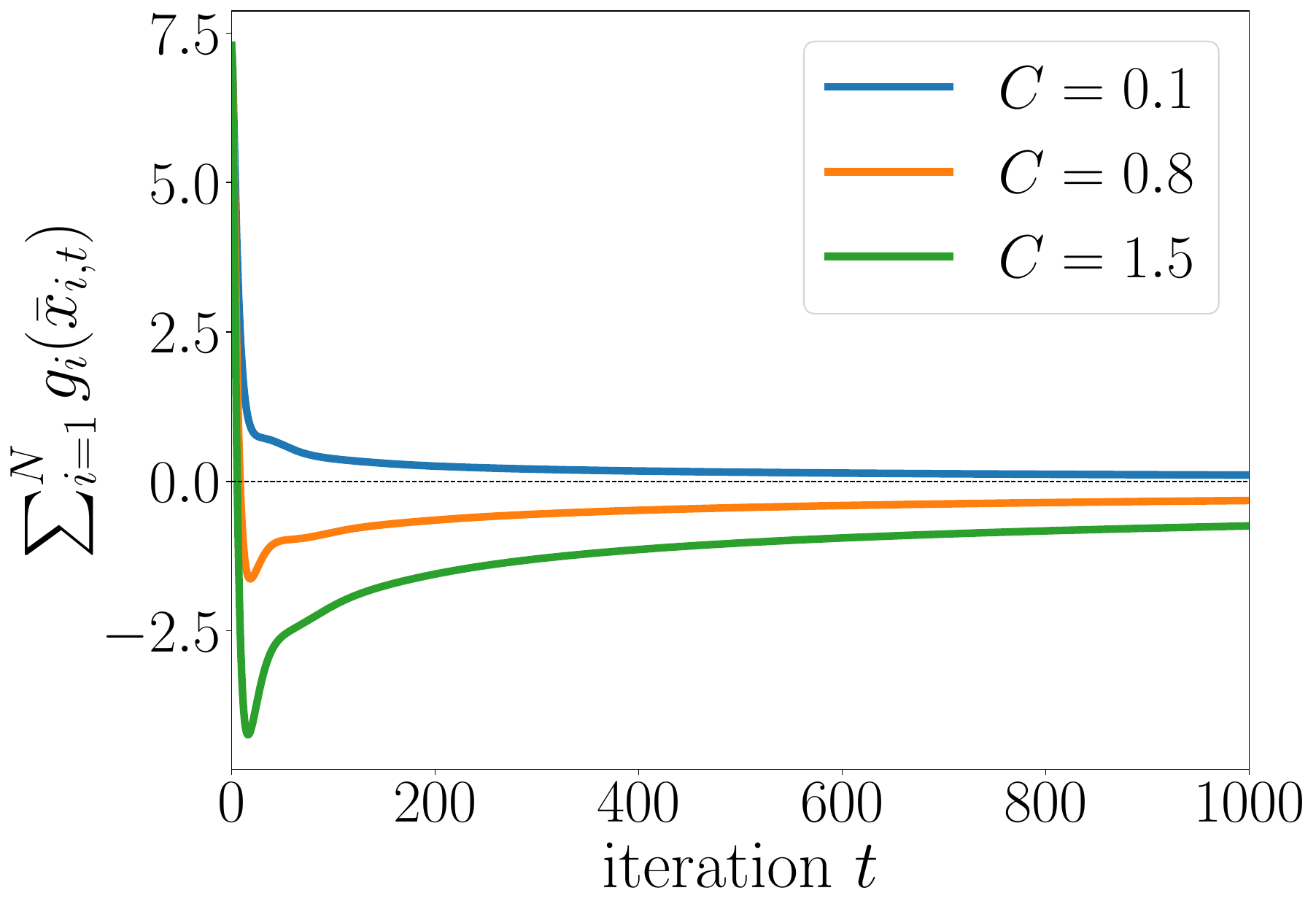}
			\label{eachlatency}
		\end{minipage}
	}	
	\caption{Effects of different buffer parameter $\gamma_t=C/\sqrt{t}$ on (a) objective error and (b) constraint violation. } 
	\label{differentC}  
\end{figure}
\begin{figure}[t] 	
	\subfigure[]
	{
		\begin{minipage}{0.466\linewidth}
			\centering 
			\includegraphics[height=0.75\linewidth,width=\linewidth]{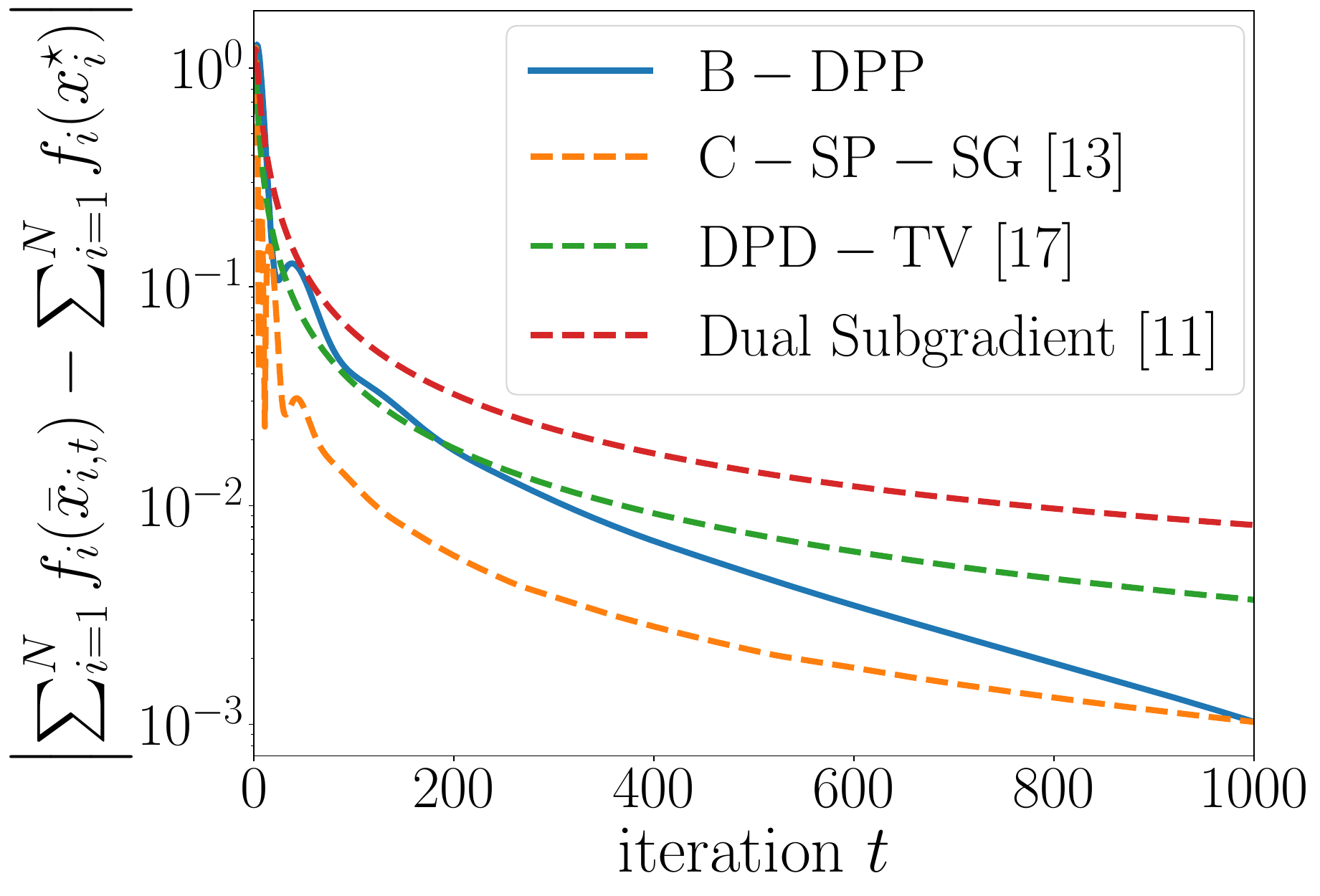}
			\label{regdifferentN}
		\end{minipage}
	}	
	\subfigure[]
	{
		\begin{minipage}{0.466\linewidth}
			\centering      
			\includegraphics[height=0.75\linewidth,width=\linewidth]{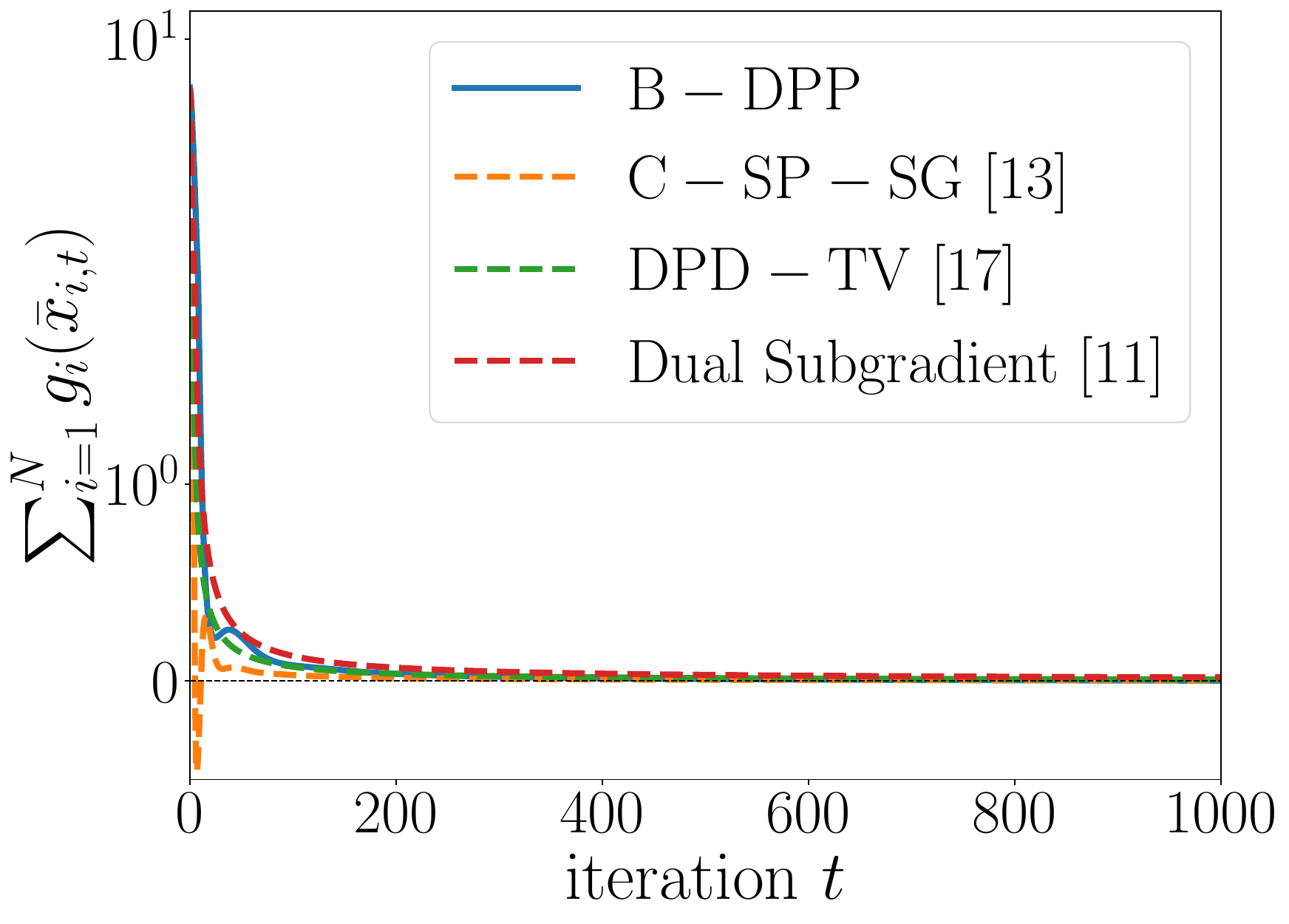}
			\label{cvdifferentN}
		\end{minipage}
	}	
	\caption{ Convergence performance of B-DPP and three alternative
		methods on (a) objective error and (b) constraint violation. } 
	\label{comparsion}  
\end{figure} 

Fig~\ref{differentC} (a) and (b) plot the objective error and the value of coupled constraint generated by B-DPP with different buffer parameter $\gamma_t=C/\sqrt{t}$, where $x_i^*$, $\forall i$ is the optimal value of problem \eqref{eq:numericalprob} calculated by CVXPY\cite{diamond2016cvxpy}. Fig~\ref{differentC} shows B-DPP achieves convergence. We can see that as $C$ increases, the objective value error increases while the constraint violation decreases. In the case of small $C$, the value of coupled constraint approaches zero from positive values. With large $C$, the value of the coupled constraint is less than zero, that is, B-DPP reaches feasibility, within 100 iterations. These facts demonstrate the theoretical results in Theorem~\ref{thm:convergencereuslts}.

We compare B-DPP with the distributed dual subgradient method \cite{falsone2017dual}, C-SP-SG\cite{mateos2016distributed}, and DPD-TV\cite{camisa2021distributed}.
For a fair comparison, the performances of the four algorithms are all evaluated over the running average of primal variables. All algorithm parameters are fine-tuned to achieve their possibly best performance, where $C=0.27$ in B-DPP, the Slater point $\hat{x}_i=1$, $\forall i$ is used in \cite{mateos2016distributed}, and the stepsizes in \cite{falsone2017dual}, \cite{camisa2021distributed} are chosen as $\frac{4.5}{t+1}$, $\frac{0.1}{(t+1)^{0.6}}$, respectively. Fig~\ref{comparsion}(a) and (b) illustrate that B-DPP exhibits similar convergence rates to C-SP-SG\cite{mateos2016distributed} and outperforms the others, confirming its competitive performance.

\section{Conclusion}\label{sec:conclusion}
We have proposed a buffering drift-plus-penalty (B-DPP) algorithm to solve distributed optimization problems with coupled inequality constraints over time-varying directed networks. By introducing a local virtual queue variable for each agent to track cumulative constraint violations and integrating a buffer parameter with the local virtual queue updates, the proposed B-DPP algorithm extends the centralized drift-plus-penalty (DPP) method to distributed sceneries and provides flexibility in solution accuracy and feasibility via adjusting the buffer parameter. B-DPP is shown to achieve $O(1/\sqrt{t})$ convergence rates in terms of objective error and constraint violation. With a proper buffer parameter, B-DPP reaches feasibility within a finite number of iterations. Its superior convergence performance is demonstrated by numerical experiments.  

\appendix
\subsection{Proof of Equation~\eqref{eq:intuitionofCDPP}} \label{prf:boundoftriangle}
Similar to \cite[Lemma~4]{wang2023distributed}, we derive
\begin{align}
	& \triangle_t\le 2NF^2+4NFp\gamma_{t+1}\!+\!4Np\gamma_{t+1}^2+N\langle\bar{\mu}_t, \gamma_{t+1}\mathbf{1}_p\rangle\displaybreak[0]\notag\\ 
	&\!+\!(F\!\!+\!\!2\sqrt{p}\gamma_{t+1})\!\sum_{i=1}^{N}\|\bar{\mu}_t\!-\!\hat{\mu}_{i,t}\|\!\!+\!\!\sum_{i=1}^{N}\langle \hat{\mu}_{i, t},~ g_i(x_{i,t+1})\rangle\!. \displaybreak[0]\label{eq:boundoftriangle}
\end{align}
Adding $V_{t+1}\sum_{i=1}^{N}f_i(x_{i,t+1})+\eta_{t+1}\sum_{i=1}^{N}\|x_{i,t+1}-x_{i,t}\|^2$ to both sides of \eqref{eq:boundoftriangle} gives \eqref{eq:intuitionofCDPP}.

\subsection{Proof of Theorem~\ref{thm:convergencereuslts}} \label{prf:convergenceresults}
The following lemmas support the proof of Theorem~\ref{thm:convergencereuslts}.

\begin{lemma}\label{lem:strongconvSit}
	Suppose Assumptions \ref{asm:probassumption}--\ref{ass: networkassumption} hold. Let $S_{i,t}(x_i,\hat{\mu}_{i})= V_{t+1}f_i(x_i)+\langle \hat{\mu}_{i}, ~ g_i(x_{i})\rangle\!+\eta_{t+1}\|x_i\!-\!x_{i,t}\|^2$. Then, 
	for any $t \ge 0$ and arbitrary $\tilde{x}_{i,t} \in X_i$,
	\begin{align}
		S_{i,t}(x_{i,t+1},\hat{\mu}_{i,t})&\le S_{i,t}(\tilde{x}_{i,t},\bar{\mu}_{t})\!-\!\eta_{t+1}\|x_{i,t+1}-\tilde{x}_{i,t}\|^2\notag \displaybreak[0]\\
		&+\!F\|\bar{\mu}_t\!-\!\hat{\mu}_{i,t}\|. \displaybreak[0] \label{eq:strongSit}
	\end{align}
\end{lemma}
\begin{proof} 
	Since $S_{i,t}(x_i,\hat{\mu}_{i,t})$ is strongly convex and $x_{i,t+1}$ is its minimum, Lemma~\ref{lem:strongconvSit} holds.
\end{proof}
\begin{lemma} \label{lem:metworkmutbound}
	Suppose Assumptions \ref{asm:probassumption}--\ref{ass: networkassumption} hold. Then, 
	for any $t \ge 1$, we have
	\begin{align}
		\sum_{i=1}^{N}\|\bar{\mu}_t\!\!-\!\!\hat{\mu}_{i,t}\| \!\le\! 2NF\!\!+\!\!\frac{N^2Fr\!+\!N^2r\sqrt{p}\gamma_1}{1\!-\!\beta}\!+\!2N\!\sqrt{p}\gamma_{t}, \displaybreak[0] \label{eq:boundofbarmuthatmuit}	
	\end{align}
	where $r=\left(1-\frac{a}{2 N^2}\right)^{-2}$, $\beta=\left(1-\frac{a}{2 N^2}\right)^{1/B}$.
\end{lemma}
\begin{proof}
	Please refer to \cite[eq.(49)]{lee2017sublinear}
\end{proof}

\begin{lemma}\label{lem:boundofnormbarmut}
	Suppose Assumptions \ref{asm:probassumption}--\ref{ass: networkassumption} hold. Let $ \tau = \lceil \sqrt{t}\rceil$ and $\delta=F+\frac{\sqrt{p}\epsilon}{2N}$. For any $t \ge 0$,
	\begin{align}
		\|\bar{\mu}_t\| \le 4\sigma\sqrt{t}\!+\!C_1\sqrt{t}+C_2, \displaybreak[0] \label{eq:boundofnormbarmut}
	\end{align}
	where $C_1=\frac{8\delta^2}{\epsilon} \log\frac{8\delta^2}{\epsilon^2}\!+\!\frac{(4NFr\sqrt{p}+2rp\epsilon)}{(1-\beta)} \!+\!\frac{8NGR+16NR^2}{\epsilon}$ and $C_2=\sqrt{p}C(2NC / \epsilon)^2\!+\!(6\!+\!(2 NC / \epsilon)^4)\delta+\frac{24NF^2}{\epsilon}\!+\!\frac{8N^2F^2r+4NFr\sqrt{p}}{\epsilon(1-\beta)}+\!\frac{24NFp\!+\!8p\epsilon}{N}$.
\end{lemma}
\addtolength{\topmargin}{0.045in}
\begin{proof}
	Consider the cases $t \geq \max \{6,(\frac{2 NC}{\epsilon})^4\}$. It follows that $t \ge 2\tau \ge (\frac{2 NC}{\epsilon})^2$. Thus, $\gamma_{t}=\frac{C}{\sqrt{t}} \le \frac{\epsilon}{2N}$.  
	By substituting  \eqref{eq:strongSit} into \eqref{eq:boundoftriangle} with $\tilde{x}_{i,t}=\hat{x}_{i}$, and adding it from $s=t, t+1, \ldots, t+\tau-1$. By following the the line of proof in \cite[Lemma~3.5]{kim2023online}, 
	we calculate the upper bounds of each term in right side of the accumulated inequality to obtain the upper bound of $\sum_{s=t}^{t+\tau-1}\!\!\triangle_s/N$. Let $\theta_t(\tau)=\frac{24NF^2}{\epsilon}\!+\!\frac{8N^2F^2r+4NFr\sqrt{p}}{\epsilon(1-\beta)}+\!\frac{24NFp\!+\!8p\epsilon}{N}+\frac{2NF\!+\!\sqrt{p}\epsilon}{N}\tau\!+\!(\frac{(4NFr\sqrt{p}+2rp\epsilon)}{(1-\beta)}+\frac{8NGR}{\epsilon}) V_{t}\!+\!\frac{16NR^2}{\epsilon\tau}\eta_t$. The upper bound of the $\sum_{s=t}^{t+\tau-1}\!\!\triangle_s/N$ is as follows
	\begin{align}
		&\frac{1}{N}\sum_{s=t}^{t+\tau-1}\!\!\triangle_s 
		\le(6F^2\!+\!\frac{2NF^2r\!+\!Fr\sqrt{p}}{1-\beta})\tau\!+\!\frac{6NFp\!+\!2p\epsilon}{N^2}\epsilon\tau\notag\\ 
		&+\frac{2NF\!+\!\sqrt{p}\epsilon}{4N^2}\epsilon\tau^2\!+\!(\frac{(2NFr\sqrt{p}\!+\!rp\epsilon)}{2N(1-\beta)}+\!\frac{2GR}{\epsilon})\epsilon\tau V_{t}\notag\displaybreak[0]\\
		&+4R^2\eta_t\!-\!\frac{\epsilon\tau}{2N}\|\bar{\mu}_t\| =  \frac{\epsilon\tau}{4N}\theta_t(\tau)-\frac{\epsilon\tau}{2N}\|\bar{\mu}_t\|, \displaybreak[0]\label{eq:boundofsumbarmuts1muts}	
	\end{align}
	which leads to
	$\|\bar{\mu}_{t+\tau}\|^2=\|\bar{\mu}_{t}\|^2+\frac{2}{N}\sum_{s=t}^{t+\tau-1}\triangle_s \le \|\bar{\mu}_{t}\|^2-\!\frac{\epsilon\tau}{N}\|\bar{\mu}_t\|+ \frac{\epsilon\tau}{2N}\theta_t(\tau)$. If $\|\bar{\mu}_t\|\ge \theta_t(\tau)$, we have $\|\bar{\mu}_{t+\tau}\| - \|\bar{\mu}_{t}\|\le-\frac{\epsilon\tau}{4N}$. Similar to \cite[Lemma~3.4(a)]{kim2023online}, we derive $|\bar{\mu}_t\|\le \sqrt{p}C(2NC / \epsilon)^2+(\frac{\sqrt{p}\epsilon}{2N}+F)t$ that results in \eqref{eq:boundofnormbarmut} when $t < \max \{6,(\frac{2 NC}{\epsilon})^4\}$.
	Let $\xi=\frac{\epsilon}{4N}$, $\tilde{r}=\frac{\xi}{4\lceil \sqrt{t}\rceil\delta^2}$, and $\rho=1-\frac{\tilde{r}\xi\tau}{2}$. By imitating the proof of \cite[Lemma~8]{kim2023online}, we obtain
	$
	e^{\tilde{r}\|\bar{\mu}_t\|}\le \rho e^{\tilde{r}\|\bar{\mu}_{t-\tau}\|}+e^{\tilde{r}\tau\delta}e^{\tilde{r}\theta_{t-\tau}(\tau)}$. Since $2\tau \ge t-(k-1)\tau \ge \tau\ge (\frac{2 NC}{\epsilon})^2$ for some $k=\lfloor \frac{t}{\tau}\rfloor \ge 2$, applying the aforementioned inequality for $s=t, t-\tau, \ldots, t-(k-2)\tau $ yields $e^{\tilde{r}\|\bar{\mu}_t\|} \le e^{\tilde{r}\sqrt{p}C(2NC / \epsilon)^2}e^{2\tilde{r}\tau\delta}e^{\tilde{r} \theta_{t}}/(1-\rho)$, thus $	\|\bar{\mu}_t\| \le \sqrt{p}C(2NC / \epsilon)^2\!+\!4\delta\sqrt{t}\!+\!\theta_{t}(\lceil\sqrt{t}\rceil)\!\!+\!\!\frac{1}{\tilde{r}} \log\frac{1}{1\!-\!\rho}$.
	Substituting $\theta_{t}(\lceil\sqrt{t}\rceil)$, $\tilde{r}$, and $\rho$ into it gives \eqref{eq:boundofnormbarmut}.
\end{proof} 
Next, we prove \eqref{eq:functionconvergence}--\eqref{eq:sqrtTconstrviolation} based on above auxiliary lemmas.
\begin{itemize}
	\item \textit{Proof of \eqref{eq:functionconvergence}}: 
	Let $\tilde{x}_{i,k}=x_{i}^*$, $\forall k \ge 0$ in Lemma~\ref{lem:strongconvSit}. With $\bar{\mu}_k \ge \mathbf{0}_p$ and $\sum_{i=1}^{N}g_i(x_i^*)\le \mathbf{0}_p$, the term $\langle \bar{\mu}_k, ~\sum_{i=1}^Ng_i(x_i^*)\rangle \le 0$. Combing \eqref{eq:boundoftriangle} with \eqref{eq:strongSit}, dividing both sides by $V_{k+1}$, and then summing it from $k=0$ to $t-1$ give
	\begin{align}
		&\sum_{k=1}^{t}\sum_{i=1}^{N}\big(f_{i}(x_{i,k})\!-\!f_{i}(x_{i}^*)\big)\le\!\sum_{k=1}^{t}\frac{2NF^2}{V_{k}}\!\!+\!\!\frac{N\sqrt{p}\gamma_{k}}{V_{k}}\|\bar{\mu}_k\| \displaybreak[0] \notag \\
		&+\!\!\sum_{k=1}^{t}\frac{4NFp\gamma_{k}\!+\!4Np\gamma_{k}^2}{V_{k}}\!+\!\frac{(2F\!+\!2\sqrt{p}\gamma_{k})}{V_{k}}\!\sum_{i=1}^{N}\|\bar{\mu}_k\!\!-\!\!\hat{\mu}_{i,k}\| \displaybreak[0] \notag\\
		&+\!\sum_{k=0}^{t-1}\frac{\eta_{k+1}}{V_{k+1}}\sum_{i=1}^N (\|x_{i,k}\!-\!x_i^*\|^2-\!\|x_{i,k+1}\!-\!x_i^*\|^2). \label{eq:sumTsumNfi}
	\end{align}
	Based on \eqref{eq:stepsizechosen}, \eqref{eq:boundofbarmuthatmuit}, and \eqref{eq:boundofnormbarmut}, we calculate upper bounds of each term on the right-hand side of \eqref{eq:sumTsumNfi}, which gives	
	\begin{align}
		\sum_{k=1}^{t}\sum_{i=1}^{N}\big(f_{i}(x_{i,k})-f_{i}(x_{i}^*)\big)\le C_f \sqrt{t}, \displaybreak[0]\label{eq:sumfibound}
	\end{align}
	where $C_f=12NF^2\!+16NPC^2\!+\!16NFpC+4N^2r(F+\sqrt{p}C)^2/(1\!-\!\beta)\!+\!2pC(C_1+C_2+4\sigma)+2R^2+2$. 
	Applying Jesen's inequality to \eqref{eq:sumfibound} gives \eqref{eq:functionconvergence}. 	
	\item \textit{Proof of \eqref{eq:sqrtTconstrviolation}}: 
	Combing Lemma~\ref{lem:cumulativeconstraintvion} with Lemma~\ref{lem:boundofnormbarmut} and using $\sum_{i=1}^{N}g_i(\bar{x}_{i,t}) \le \sum_{k=1}^{t}\sum_{i=1}^{N}g_i(\bar{x}_{i,k})/t $ give
	\begin{align}
		\sum_{i=1}^{N}g_i(\bar{x}_{i,t})
		\le \frac{N(4\sigma\!+\!C_1\!-\!C)\mathbf{1}_p}{\sqrt{t}} +\frac{NC_2\mathbf{1}_p}{t}, \label{eq:boundofconsViolation}
	\end{align}
which implies that \eqref{eq:sqrtTconstrviolation} holds with $C_g=N(4\sigma\!+\!C_1\!+\!C_2\!-\!C)$, where $\sigma$, $C_1$, and $C_2$ are given in Lemma~\ref{lem:boundofnormbarmut}.
	\item \textit{Proof of \eqref{eq:zeroconstriant}}: 
	Let $C_0=4\sigma+C_1+1$, where $C_1$ is given in Lemma~\ref{lem:boundofnormbarmut}. For any $C \ge C_1$, $t \ge t_1$, \eqref{eq:boundofconsViolation} indicates $\sum_{i=1}^{N}g_i(\bar{x}_{i,t}) \le \mathbf{0}_p$, where $t_1$ can be derived from \eqref{eq:boundofconsViolation}. 
\end{itemize}


\addtolength{\textheight}{-3cm} 

\bibliographystyle{IEEEtran}
\bibliography{lcssref}

\end{document}